\documentclass[11pt,reqno]{amsart}
\usepackage{amsmath}
\usepackage{amsthm}
\usepackage{amssymb}
\usepackage{enumerate}
\usepackage{amscd}
\usepackage{color}
\usepackage{diagrams}
\usepackage{mathtools}
\usepackage{graphicx}
\usepackage[all, cmtip]{xy}
\usepackage{soul}
\usepackage{esint}
\usepackage{hyperref}
\hypersetup{pdftex,colorlinks=true,allcolors=blue}
\usepackage{hypcap}
\usepackage[titletoc]{appendix}

\usepackage{amsrefs}

\theoremstyle{plain}
\newtheorem{lemma}{Lemma}
\newtheorem{prop}[lemma]{Proposition}
\newtheorem{coro}[lemma]{Corollary}

\newtheorem{rema}[lemma]{Remark}
\newtheorem{remark}[lemma]{Remark}
\newtheorem{thm-Intro}{Theorem} 
\newtheorem{cor-Intro}{Corollary} 

\numberwithin{equation}{section}

\mathtoolsset{showonlyrefs=true}

\begin{document}
\title[Octonionic Hopf fibration and subelliptic heat kernel]{The subelliptic heat kernel of the octonionic Hopf fibration}

 \author{Fabrice Baudoin}
 \thanks{F.B is partially funded by the Simons Foundation and NSF grant DMS-1901315.}
 \author{Gunhee Cho}
 
\address{Department of Mathematics\\
University of Connecticut\\
196 Auditorium Road,
Storrs, CT 06269-3009, USA}

\email{fabrice.baudoin@uconn.edu}
\email{gunhee.cho@uconn.edu}

\begin{abstract} 
We study the sub-Laplacian of the $15$-dimensional unit sphere which is obtained by lifting with respect to the Hopf fibration the Laplacian of the octonionic projective space. We obtain in particular explicit formulas for its heat kernel and deduce an expression for the Green function of a related sub-Laplacian. As a byproduct we also obtain the spectrum of the sub-Laplacian, the small-time asymptotics of the heat kernel and explicitly compute the sub-Riemannian distance. 	 
\end{abstract}

\maketitle

\tableofcontents

\section{Introduction}

R. Escobales proved in \cite{Escobales}  that, up to equivalence, the only Riemannian submersions with connected totally geodesic fibers from a unit sphere are given by:
\begin{enumerate}
\item \underline{The complex Hopf fibrations}:
\[
\mathbb{S}^1 \hookrightarrow \mathbb{S}^{2n+1} \rightarrow {\mathbb{CP}}^n.
\]
\item \underline{The quaternionic Hopf fibrations}:
\[
\mathbb{S}^3 \hookrightarrow \mathbb{S}^{4n+3} \rightarrow {\mathbb{HP}}^n.
\]
\item \underline{The octonionic Hopf fibration}:
\[
\mathbb{S}^7 \hookrightarrow \mathbb{S}^{15} \rightarrow {\mathbb{OP}}^1.
\]
\end{enumerate}

The thorough  study of the horizontal Laplacians and associated heat kernels of the complex and quaternionic Hopf fibrations was respectly done in \cite{FBJW13} and \cite{FBJW14}.
The main goal of the present paper is to complete the picture and study  the geometry, the horizontal Laplacian and the horizontal heat kernel of the octonionic Hopf fibration which is the only remaining case (3).

 

The  horizontal Laplacian of the fibration is the lift on $\mathbb{S}^{15}$ of the Laplace-Beltrami operator of $\mathbb{OP}^1$. However, unlike the submersion $ \mathbb{S}^{2n+1} \rightarrow {\mathbb{CP}}^n$ and the submersion  $ \mathbb{S}^{4n+3} \rightarrow {\mathbb{HP}}^n$ which have been considered in \cite{FBJW13} and \cite{FBJW14}, the fibre $\mathbb{S}^7$ does not admit a Lie group structure, it seems therefore  non-trivial to obtain the explicit description of horizontal Laplacian $L$.  This horizontal Laplacian also appears the sub-Laplacian of a canonical H-type sub-Riemannian structure on $\mathbb{S}^{15}$, see Table 3 in \cite{BGMR}. For this reason, in the sequel the horizontal Laplacian will be referred to as the sub-Laplacian.

Let us briefly describe our main results. Due to the cylindrical symmetries of the fibration, the heat kernel of the sub-Laplacian only depends on two variables: the variable $r$ which is a radial on $\mathbb{OP}^1$ and the variable $\eta$ which is a radial coordinate on the fiber $\mathbb{S}^7$. We prove that in these coordinates, the radial part of the sub-Laplacian writes 
\begin{equation*}
\frac{{\partial}^2}{\partial {r}^2}+(7\cot r-7\tan r)\frac{{\partial}}{\partial {r}}+{\tan^2 r}\left(\frac{{\partial}^2}{\partial {\eta}^2}+6\cot \eta\frac{{\partial}}{\partial {\eta}}\right).
\end{equation*}
As a consequence of this expression for the sub-Laplacian, we are able to derive two expressions for the heat kernel:

(1) \underline{A Minakshisundaram-Pleijel spectral expansion}:
For $r\in [0,\frac{\pi}{2}),\eta\in [0,\pi)$, we have:
\begin{equation*}
p_t(r,\eta)=\sum_{m=0}^{\infty}\sum_{k=0}^{\infty}\alpha_{k,m} \frac{\Gamma(7/2)}{\sqrt{\pi}\Gamma(3)}\int_{0}^{\pi}{(\cos \eta + \sqrt{-1}\sin \eta \cos \varphi)}^m{\sin^5 \varphi}d \varphi,
\end{equation*}
\begin{equation*}
\times e^{-(8m+4k(k+m+7))t}\cos^m r P^{3,m+3}_k(\cos 2r),
\end{equation*}
where $\alpha_{k,m}=\frac{96}{\pi^{8}}(m+3)(2k+m+7)\left(\begin{array}{c}
k+m+6 \\ k+m+3
\end{array}\right)\left(\begin{array}{c}
m+5 \\ m
\end{array}\right)$ and

\begin{equation*}
P^{3,m+3}_k(x)=\frac{(-1)^k}{2^k k! (1-x)^{3}(1+x)^{m+3}}\frac{{d}^k}{d {x}^k}((1-x)^{k+3}(1+x)^{m+3+k})
\end{equation*} 
is a Jacobi polynomial.  In particular, the spectrum of ${-L}$ is given by $$\{ 4k(k+m+7)+8m : m,k\geq 0 \}.$$

This spectral expansion is useful to study the long-time behavior of the heat kernel but it might be difficult to use in the study of small-time asymptotics. In order to derive small-time asymptotics of the heat kernel, we give another analytic expression for $p_t(r,\eta)$.

(2) \underline{An integral representation}:

For $r\in [0,\frac{\pi}{2}),\eta\in [0,\pi)$, we have:
\begin{equation}\label{eq:8}
p_t(r,\eta)=\frac{48e^{15t}}{{\pi}^2\sqrt{\pi t}\cos^2 r}\int_{0}^{\infty}\frac{g_t(\eta,y)}{t^2} e^{{-}\frac{y^2+{\eta}^2}{4t}}q_t(\cos r \cosh y)\sinh y dy,
\end{equation}
where 
\begin{align*}
g_t(\eta,y)&=\csc ^3(\eta)\left( \cos \left(\frac{\eta y}{2 t}\right)2 y (\eta-3 t \cot \eta)\right.\\
&+\left.\sin \left(\frac{\eta y}{2 t}\right) \left(8 t^2 \cot ^2 \eta+4 t^2 \csc ^2 \eta-6 t \eta \cot \eta+2 t+\eta^2-y^2\right)\right).
\end{align*}

and $q_t$ is the Riemannian heat kernel on $\mathbb{S}^{11}$. We obtain this formula by comparing the subelliptic heat kernel of the sub-Laplacian associated to the quaternionic Hopf fibration. From this formula we are able to deduce the fundamental solution of the operator $-L+40$; It is given in cylindrical coordinates by

\begin{equation*}
G(r,\eta)=\frac{C}{\left(-2 \cos r \cos \eta+\cos ^2 r+1\right)^5},
\end{equation*}
where $C>0$ is an explicit constant.

Furthermore, we also derive three different behaviors of the small-time asymptotics of the heat kernel: on the diagonal, on the vertical cut-locus, and outside of the cut-locus. As an interesting by-product of this small-time asymptotics we obtain an explicit formula for the sub-Riemannian distance on the octonionic unit sphere. In particular, we obtain that the sub-Riemannian diameter of the octonionic fibration is given by $\pi$.

\section{Preliminary: The geometry of the octonionic Hopf fibration }

In this section, we  describe the octonionic Hopf fibration. We refer to \cite{LOMPPPVV14} for additional and complementary details.

We consider the non-associative (but alternative) division algebra of octonions which is described by 
\begin{equation*}
\mathbb{O}=\left\{x=\sum_{j=0}^{7}x_j e_j, x_j\in \mathbb{R} \right\},
\end{equation*}
where the multiplication rules are given by 
\begin{equation*}
e_i e_j=e_j \text{  if  } i=0,
\end{equation*}
\begin{equation*}
e_i e_j=e_i \text{  if  } j=0,
\end{equation*}
\begin{equation*}
e_i e_j=-\delta_{ij}e_0+\epsilon_{ijk}e_k \text{  otherwise},
\end{equation*}
where $\delta_{ij}$ is the Kronecker delta and $\epsilon_{ijk}$ is the completely antisymmetric tensor with value $1$ when $ijk = 123, 145, 176, 246, 257, 347, 365$.


The octonionic norm is defined for $x \in \mathbb{O}$ by
\begin{equation*}
\|x\|^2=\sum_{j=0}^{7}x^2_j.
\end{equation*}

The  unit sphere in $\mathbb{O}^2$  is given by
\begin{equation*}
\mathbb{S}^{15}=\{(x,y)\in\mathbb{O}^2, \|x \|^2+\|y\|^2=1 \}.
\end{equation*}

 We have a Riemannian submersion $\pi : \mathbb{S}^{15} \rightarrow \mathbb{OP}^1$, given by $(x,y) \mapsto [x:y]$, where $[x:y]=y^{-1}x$. Then the vertical distribution $\mathcal{V}$ and the horizontal distribution $\mathcal{H}$  of $T\mathbb{S}^{15}$ are defined by $\ker d\pi$ and the orthogonal complement of $\mathcal{V}$ respectively so that $T\mathbb{S}^{15}=\mathcal{H} \oplus \mathcal{V}$. Note that $\pi : \mathbb{S}^{15} \rightarrow \mathbb{OP}^1$ has totally geodesic fibers, and for each $b\in \mathbb{OP}^1$, the  fiber ${\pi}^{-1}(\{b\})$ is isometric to $\mathbb{S}^7$ with the standard sphere metric $g_{\mathbb{S}^7}$. 
 
 This submersion $\pi$ yields the octonionic Hopf fibration:
\begin{equation*}
\mathbb{S}^7 \hookrightarrow \mathbb{S}^{15} \rightarrow \mathbb{OP}^1.
\end{equation*}

The submersion $\pi$ also yields an $H$-type foliation structure in the sense of \cite{BGMR} and thus $\mathbb{S}^{15}$ carries a sub-Riemannian structure inherited from this foliation.

Unlike the quaternionic fibration, see Section 3.1 in \cite{FBJW14}, the octonionic fibration is not projectable. Indeed, In addition to the octonionic Hopf fibration 
\[
\mathbb{S}^7 \hookrightarrow \mathbb{S}^{15} \rightarrow {\mathbb{OP}}^1
\]
that is considered in this paper, on $\mathbb{S}^{15}$ one can also consider the complex Hopf fibration
\[
\mathbb{S}^1 \hookrightarrow \mathbb{S}^{15} \rightarrow {\mathbb{CP}}^7
\]
and the quaternionic one
\[
\mathbb{S}^3 \hookrightarrow \mathbb{S}^{15} \rightarrow {\mathbb{HP}}^3.
\]

Similarly to the projection of fibration procedure explained in Section 3.1 in \cite{FBJW14}, this would potentially yield two commutative diagrams:

\begin{diagram}\label{diag1}
  & & \mathbb{S}^1 & & \\
  & \ldTo & \dTo & & \\
\mathbb{S}^7& \rTo &\mathbb{S}^{15} & \rTo & \mathbb{OP}^1 \\
 \dTo & &\dTo & \ruTo & \\
 \mathbb{CP}^3 & \rTo & \mathbb{CP}^{7} & & 
\end{diagram}

and 

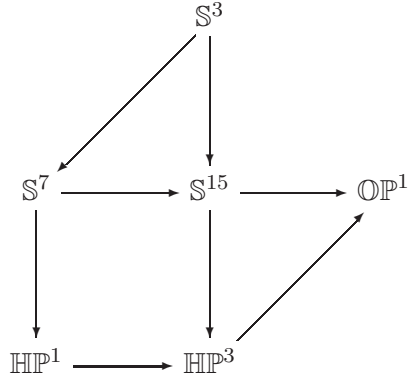
\begin{diagram}\label{diag2}
  & & \mathbb{S}^3 & & \\
  & \ldTo & \dTo & & \\
 \mathbb{S}^7 & \rTo &\mathbb{S}^{15} & \rTo & \mathbb{OP}^1 \\
 \dTo & &\dTo & \ruTo & \\
 \mathbb{HP}^1 & \rTo & \mathbb{HP}^{3} & & 
\end{diagram}

However, unlike the quaternionic case \cite{FBJW14}, those diagrams actually do not exist. Indeed,  in the first diagram the submersion
\[
  \mathbb{CP}^{7} \to \mathbb{OP}^1
\]

does not exist, see \cite{Ucci} and  \cite{AR85} page 258.  In the second diagram the submersion
\[
 \mathbb{HP}^{3}  \to \mathbb{OP}^1
\]
does not exist, see \cite{Ucci} and \cite{Escobales2}.
\section{Cylindric coordinates and radial part of the sub-Laplacian}

The geometry of the octonionic Hopf fibration shares many properties with the geometry of the complex and quaternionic Hopf fibration. Similarly,  the analysis of the octonionic sub-Laplacian on $\mathbb{S}^{15}$  parallels the ones of the complex and quaternionic sub-Laplacians which were undertaken in \cite{FBJW13, FBJW14}.

The sub-Laplacian $L$ on $\mathbb{S}^{15}$ we are interested in  is the horizontal Laplacian of the Riemannian submersion $\pi:\mathbb{S}^{15} \to \mathbb{OP}^1$, i.e the horizontal lift of the Laplace-Beltrami operator of $\mathbb{OP}^1$. It can be written as
\[
L=\triangle_{\mathbb{S}^{15}}-\triangle_\mathcal{V},
\]
where $\triangle_\mathcal{V}$ is the vertical Laplacian. Since the fibers of $\pi$ are totally geodesic, we note that $\triangle_{\mathbb{S}^{15}}$ and $\triangle_\mathcal{V}$ are commuting operators (see \cite{BGLB82}). Since the horizontal distribution of the octonionic fibration is bracket-generating, $L$ is a hypoelliptic operator. We note that the sub-Riemannian structure induced by the fibration is even fat; it is actually an H-type sub-Riemannian structure in the sense of \cite{BGMR}, see Remark 2.16 in \cite{BGMR}.

To study $L$, we introduce a  set of coordinates that reflect the cylindrical symmetries of the octonionic unit sphere with respect to the octonionic Hopf fibration.  Take local coordinates $w\in {\mathbb{OP}}^1\backslash \{\infty \}$ and $(\theta_1, ..., \theta_7)\in \mathbb S^7$, where $w$ is the  inhomogeneous coordinate on ${\mathbb{OP}}^1\backslash \{\infty \}$ given by $w=y^{-1} x$, where $x,y\in \mathbb{O}$. Consider the  pole $p=(1,0,\cdots,0) \in \mathbb{S}^7$, take $Y_1, ..., Y_7$ to be an orthonormal frame of $T_p \mathbb{S}^7$ and denote $\exp_p$ the Riemannian exponential map at $p$ on $\mathbb{S}^7$. Then the cylindrical coordinates we work with are given by 
\begin{equation*}
(w,\theta_1, ..., \theta_7) \mapsto \left(\frac{ \exp_p(\sum_{i=1}^{7}\theta_i Y_i) w}{\sqrt{1+\|w\|^2}},\frac{ \exp_p(\sum_{i=1}^{7}\theta_i Y_i)}{\sqrt{1+\|w\|^2}}\right)\in \mathbb{S}^{15}.
\end{equation*}
This parametrizes the set  $\Omega=\{ (x,y) \in \mathbb{S}^{15}, y  \neq 0,   \frac{y}{\| y \|}  \neq q\}$ where $q$ denotes the antipodal point to $p$.

A key property of those coordinates is that since the octonionic multiplication is alternating one has for the submersion $\pi : \mathbb{S}^{15} \rightarrow \mathbb{OP}^1$
\[
\pi \left(\frac{ \exp_p(\sum_{i=1}^{7}\theta_i Y_i) w}{\sqrt{1+\|w\|^2}},\frac{ \exp_p(\sum_{i=1}^{7}\theta_i Y_i)}{\sqrt{1+\|w\|^2}}\right)=w.
\]
Thus $\theta_1,\cdots, \theta_7$ are fiber coordinates for the octonionic Hopf fibration. 

The fiber $\mathbb S^7$ and the base space $\mathbb{OP}^1$ are both rank one symmetric spaces, thus are two point homogeneous spaces (see chapter 3 in \cite{MR496885}). As a consequence, the heat kernel will actually only depend on two coordinates: a radial coordinate on $\mathbb S^7$ and a radial coordinate on $\mathbb{OP}^1$. We can make this precise as follows.


%
%
Let us denote by $\psi$ the map on $\Omega $  such that 
\begin{equation*}
\psi \left(\frac{ \exp_p(\sum_{i=1}^{7}\theta_i Y_i) w}{\sqrt{1+\|w\|^2}},\frac{ \exp_p(\sum_{i=1}^{7}\theta_i Y_i)}{\sqrt{1+\|w\|^2}}\right)= \left( r, \eta \right),
\end{equation*}
where $r=\arctan \|w\|\in [0,\pi/2), \eta=\| \theta \| \in [0,\pi)$.

The variable $r$ can be interpreted as the Riemannian distance on $\mathbb{OP}^1$ from the point $w=0$. The variable $\eta$ can be interpreted  as the Riemannian distance from $p$ on $\mathbb{S}^7$.  We note that geometrically the boundary of $\Omega$ corresponds to the boundary values $r=\pi/2$, $\eta=\pi$.

We denote by $\mathcal{D}$ the space and smooth and compactly supported functions on $[0,\pi/2) \times [0,\pi)$. Then the radial part of $L$ is defined as the operator $\widetilde{L}$  such that for any $f\in \mathcal{D}$, we have
\begin{equation*}
L(f \circ \psi)=(\widetilde{L}f)\circ \psi.
\end{equation*}

We now compute $\widetilde{L}$ in cylindric coordinates.

\begin{prop}
	The radial part of the sub-Laplacian on $\mathbb{S}^{15}$ is given in the coordinates $(r,\eta)$ by the operator
\begin{equation}\label{eq:5}
\widetilde{L}=\frac{{\partial}^2}{\partial {r}^2}+(7\cot r-7\tan r)\frac{{\partial}}{\partial {r}}+{\tan^2 r}\left(\frac{{\partial}^2}{\partial {\eta}^2}+6\cot \eta\frac{{\partial}}{\partial {\eta}}\right).
\end{equation}	

\begin{proof}
The idea is to compute first the radial part $\tilde{\triangle}_{\mathbb{S}^{15}}$ of the Laplace-Beltrami operator on $\mathbb{S}^{15}$ and then use the fact that $L=\triangle_{\mathbb{S}^{15}}-\triangle_\mathcal{V}$. Since the octonionic Hopf fibration defines a totally geodesic submersions with base space $\mathbb{OP}^1$ and fiber $\mathbb{S}^{7}$, the Riemannian metric $g_{\mathbb{S}^{15}}$ on $\mathbb{S}^{15}$ is locally given by a warped metric $g_{\mathbb{S}^{7}}\oplus fg_{\mathbb{OP}^{1}}$ between the Riemannian metric $g_{\mathbb{OP}^{1}}$ of $\mathbb{OP}^1$ and the Riemannian metric $g_{\mathbb{S}^{7}}$ on $\mathbb{S}^{7}$, where $f$ is a smooth and positive function on $\mathbb{S}^{7}$;  See 9.11 in \cite{MR867684} for a discussion of warped products in the context of submersions.

As Riemannian manifolds, $\mathbb{OP}^1$ and $\mathbb{S}^{7}$ are compact rank one symmetric spaces. General formulas for the radial parts of Laplacians on rank one symmetric spaces are well-known (for example, see for instance chapter 3 in \cite{MR496885},  but also \cite{SH65} and \cite{MR754767}).
In particular, the radial part of the  Laplace-Beltrami operator on $\mathbb{OP}^1$ is
\[
\frac{{\partial}^2}{\partial {r}^2}+(7\cot r-7\tan r)\frac{{\partial}}{\partial {r}}
\]
and the radial part of the Laplace-Beltrami operator on $\mathbb{S}^{7}$ is
\[
\frac{{\partial}^2}{\partial {\eta}^2}+6\cot \eta\frac{{\partial}}{\partial {\eta}}
\]

One deduces that 
\[
\tilde{\triangle}_{\mathbb{S}^{15}}=\frac{{\partial}^2}{\partial {r}^2}+(7\cot r-7\tan r)\frac{{\partial}}{\partial {r}} +g(r) \left(\frac{{\partial}^2}{\partial {\eta}^2}+6\cot \eta\frac{{\partial}}{\partial {\eta}} \right)
\]
for some function $g$ to be computed.
One can compute $g$ by observing that on $\mathbb{S}^{15}$  the Riemannian distance $\delta$ from the point with octonionic coordinates $(0,1)=(0,p)\in \mathbb{O}^2$ to the point
\[
\left(\frac{ \exp_p(\sum_{i=1}^{7}\theta_i Y_i) w}{\sqrt{1+\|w\|^2}},\frac{ \exp_p(\sum_{i=1}^{7}\theta_i Y_i)}{\sqrt{1+\|w\|^2}}\right).
\]
is given by
\[
\cos \delta = \cos r \cos \eta
\]
because the right and left hand side of the above equality are both the 9th Euclidean coordinate of
\[
\left(\frac{ \exp_p(\sum_{i=1}^{7}\theta_i Y_i) w}{\sqrt{1+\|w\|^2}},\frac{ \exp_p(\sum_{i=1}^{7}\theta_i Y_i)}{\sqrt{1+\|w\|^2}}\right),
\]
where we use the fact that
\[
\frac{\cos \eta}{\sqrt{1+\|w\|^2}} =\frac{\cos \eta}{\sqrt{1+\tan^2 r}}=\cos r \cos \eta.
\]
From the formula for the radial part of Laplacian on ${\mathbb{S}^{15}}$ starting from the north pole,  we can compute
\begin{align*}
\tilde{\triangle}_{\mathbb{S}^{15}} (\cos \delta ) &=\left(\frac{\partial^2}{\partial \delta^2}+14\cot{\delta}\frac{\partial}{\partial \delta}\right) \cos \delta \\
&=-15 \cos \delta.
\end{align*}
Using the other representation of $\tilde{\triangle}_{\mathbb{S}^{15}}$, one deduces
\[
\left(\frac{{\partial}^2}{\partial {r}^2}+(7\cot r-7\tan r)\frac{{\partial}}{\partial {r}} +g(r) \left(\frac{{\partial}^2}{\partial {\eta}^2}+6\cot \eta\frac{{\partial}}{\partial {\eta}} \right) \right)\cos r \cos \eta =-15 \cos r \cos \eta.
\]
After a straightforward computation, this yields $g(r)=\frac{1}{\cos^2 r}$ and therefore
\[
\tilde{\triangle}_{\mathbb{S}^{15}}=\frac{{\partial}^2}{\partial {r}^2}+(7\cot r-7\tan r)\frac{{\partial}}{\partial {r}} +\frac{1}{\cos^2 r} \left(\frac{{\partial}^2}{\partial {\eta}^2}+6\cot \eta\frac{{\partial}}{\partial {\eta}} \right).
\]
Finally, to conclude, one notes that the sub-Laplacian $L$ is given by the difference between the Laplace-Beltrami operator of $\mathbb{S}^{15}$ and the vertical Laplacian.  Therefore,
\begin{align*}
\tilde{L} &=\tilde{\triangle}_{\mathbb{S}^{15}}- \left(\frac{{\partial}^2}{\partial {\eta}^2}+6\cot \eta\frac{{\partial}}{\partial {\eta}} \right)\\
&= \frac{{\partial}^2}{\partial {r}^2}+(7\cot r-7\tan r)\frac{{\partial}}{\partial {r}}+{\tan^2 r}\left(\frac{{\partial}^2}{\partial {\eta}^2}+6\cot \eta\frac{{\partial}}{\partial {\eta}}\right).
\end{align*}
%
%
%
\end{proof}
	
\end{prop}

\begin{rema}
As a consequence of the previous result, we can check that the Riemannian measure of $\mathbb{S}^{15}$ in the coordinates $(r,\eta)$, which is the symmetric and invariant measure for $\tilde{L}$ is given by 
\begin{equation}\label{eq:5}
d\overline{\mu}=\frac{56{\pi}^7}{\Gamma(8)}\sin^7 r \cos^7r\sin^6 \eta dr d\eta,
\end{equation}
where the normalization is chosen in such a way that
\begin{equation*}
\int_{0}^{\pi}\int_{0}^{\frac{\pi}{2}}d\overline{\mu}=\mathbf{Vol} (\mathbb{S}^{15})=\frac{2\pi^8}{\Gamma(8)}.
\end{equation*}
\end{rema}

\section{Spectral expansion of the subelliptic heat kernel}

In this section, we derive the spectral decomposition of the subelliptic heat kernel of the heat semigroup $P_t=e^{tL}$ issued from the north pole (i.e. the point with octonionic coordinates $(0,p)=(0,1) \in \mathbb{O}^2$). Notice that due to the cylindric symmetry, the heat kernel that we denote $p_t(r,\eta)$ will only depend on the coordinates $(r,\eta)$. We first prove the following  spectral expansion theorem.

We will need the Jacobi polynomial
\begin{equation*}
P^{3,m+3}_k(x)=\frac{(-1)^k}{2^k k! (1-x)^{3}(1+x)^{m+3}}\frac{{d}^k}{d {x}^k}((1-x)^{k+3}(1+x)^{m+3+k}).
\end{equation*}

\begin{prop}\label{prop:2}
For $t>0$, $r\in [0,\frac{\pi}{2})$, $\eta \in [0,\pi)$, the subelliptic kernel is given by 
\begin{equation*}
p_t(r,\eta)=\sum_{m=0}^{\infty}\sum_{k=0}^{\infty}\alpha_{k,m} h_m(\eta)e^{-4(2m+k(k+m+7))t}(\cos r)^m P^{3,m+3}_k(\cos 2r),
\end{equation*}
where $\alpha_{k,m}=\frac{96}{\pi^{8}}(m+3)(2k+m+7)\left(\begin{array}{c}
k+m+6 \\ k+m+3
\end{array}\right)\left(\begin{array}{c}
m+5 \\ m
\end{array}\right)$ , and
\[
h_m(\eta)=\frac{\Gamma(7/2)}{\sqrt{\pi}\Gamma(3)}\int_{0}^{\pi}{(\cos \eta + \sqrt{-1}\sin \eta \cos \varphi)}^m{\sin^5 \varphi}d \varphi
\]
 is the normalized eigenfunction of $\widetilde{\triangle}_{S^7}=\frac{{\partial}^2}{\partial {\eta}^2}+6\cot \eta\frac{{\partial}}{\partial {\eta}}$ which is associated to the eigenvalue $-m(m+6)$.
\begin{proof}
We expand $p_t(r,\eta)$ in spherical harmonics as follows,
\begin{equation*}
p_t(r,\eta)=\sum_{m=0}^{\infty}h_m(\eta)\phi_m(t,r),
\end{equation*}
where $h_m(\eta)$ is the eigenfunction of $\widetilde{\triangle_{S^7}}=\frac{{\partial}^2}{\partial {\eta}^2}+6\cot \eta\frac{{\partial}}{\partial {\eta}}$ which is associated to the eigenvalue $-m(m+6)$. More precisely, $h_m(\eta)$ is given by
\begin{equation*}
h_m(\eta)=\frac{\Gamma(7/2)}{\sqrt{\pi}\Gamma(3)}\int_{0}^{\pi}{(\cos \eta + \sqrt{-1}\sin \eta \cos \varphi)}^m{\sin^5 \varphi}d \varphi
\end{equation*}
(for example, see proposition 9.4.4 in \cite{JF08}).

To determine $\phi_m$, we use $\frac{{\partial}}{\partial {t}}p_t=\widetilde{L} p_t$, and find 
\begin{equation*}
\frac{{\partial}\phi_m}{\partial {t}} = \frac{{\partial}^2\phi_m}{\partial {r}^2}+(7\cot r-7\tan r)\frac{{\partial}\phi_m}{\partial {r}}-m(m+6){\tan^2 r}\phi_m. 
\end{equation*}
Let $\phi_m:=e^{-8mt}(\cos r)^m \psi_m$. This substitution gives 
\begin{equation*}
\frac{{\partial}\psi_m}{\partial {t}} = \frac{{\partial}^2\psi_m}{\partial {r}^2}+(7\cot r-(2m+7)\tan r)\frac{{\partial}\psi_m}{\partial {r}}. 
\end{equation*}
Letting $\psi_m(t,r):=g_m(t,\cos 2r)$. Then the previous equation becomes 
\begin{equation*}
\frac{{\partial}g_m}{\partial {t}} =4(1-x^2) \frac{{\partial}^2 g_m}{\partial {x}^2}+4((m-(m+8)x)\frac{{\partial}g_m}{\partial {x}}. 
\end{equation*}
We get $\frac{{\partial}g_m}{\partial {t}}=4\Psi_m(g_m)$, where 
\begin{equation*}
\Psi_m=(1-x^2)\frac{{\partial}^2}{\partial {x}^2}+((m-(m+8)x)\frac{{\partial}}{\partial {x}}.
\end{equation*}
Note that the equation 
\begin{equation*}
\Psi_m(g_m)+k(k+m+7)g_m=0
\end{equation*}
is a Jacobi differential equation for all $k \geq 0$. We denote the eigenvector of $\Psi_m$ corresponding to the eigenvalue $-k(k+m+7)$ by $P^{3,m+3}_k(x)$, which is given by 
\begin{equation*}
P^{3,m+3}_k(x)=\frac{(-1)^k}{2^k k! (1-x)^{3}(1+x)^{m+3}}\frac{{d}^k}{d {x}^k}((1-x)^{k+3}(1+x)^{m+3+k}).
\end{equation*}
(for the details about Jacobi differential equations, for example, see \cite{NDMZ09}, appendix in \cite{FBJW16} and the references therein for further details). At the end we can therefore write the spectral decomposition of $p_t$ as 
\begin{equation*}
p_t(r,\eta)=\sum_{m=0}^{\infty}\sum_{k=0}^{\infty}\alpha_{k,m} h_m(\eta)e^{-4(k(k+m+7)+2m)t}\cos^m r P^{3,m+3}_k(\cos 2r). 
\end{equation*}
where the constants $\alpha_{k,m}$'s have to be determined by considering the initial condition. 

Note that $((1+x)^{\frac{m+3}{2}} P^{3,m+3}_k(x))_{k\geq 0}$ is a complete orthogonal basis of the Hilbert space $L^2([-1,1],(1-x)^{3}dx)$, more precisely 
\begin{equation*}
\int_{-1}^{1}P^{3,m+3}_k(x)P^{3,m+3}_l(x)(1-x)^{3}(1+x)^{m+3}dx
\end{equation*}
\begin{equation*}
=\frac{2^{m+7}}{2k+m+7}\frac{\Gamma(k+4)\Gamma(k+m+4)}{\Gamma(k+m+7)\Gamma(k+1)}\delta_{kl}. 
\end{equation*}

On the other hand, $(h_m(\eta))_{m\ge 0}$ are the eigenfunctions of the self adjoint operator $\widetilde{\triangle_{S^7}}$ and thus form a complete orthonormal basis of $L^2 ( [0, \pi], (\sin \eta)^6  d\eta)$.

Thus, using the fact that $\left( \frac{1+\cos 2r}{2}\right)^{1/2} =\cos r$, for a smooth function $f(r,\eta)$, we can write 
\begin{equation*}
f(r,\eta)=\sum_{m=0}^{\infty}\sum_{k=0}^{\infty}\beta_{k,m} h_m(\eta)\cos^m r P^{3,m+3}_k(\cos 2r)
\end{equation*}
where the $\beta_{k,m}$'s are constants. We obtain then 
\begin{equation*}
f(0,0)=\sum_{m=0}^{\infty}\sum_{k=0}^{\infty}\beta_{k,m}P^{3,m+3}_k(1)
\end{equation*}
and we observe that $P^{3,m+3}_k(1)=\left(\begin{array}{c}
3+k \\ k
\end{array}\right)$. From \eqref{eq:5}, the measure $d\overline{\mu}$ is given in cylindric coordinates by 
\begin{equation*}
d\overline{\mu} = \frac{56{\pi}^7}{\Gamma(8)}(\sin r)^{7}(\cos r)^{7}(\sin \eta)^{6}dr d\eta.
\end{equation*}
Moreover, we have
\begin{align*}
&\int_{0}^{\pi}\int_{0}^{\frac{\pi}{2}}p_t(r,\eta)f(-r,-\eta)d\overline{\mu}\\
&=\frac{56{\pi}^7}{\Gamma(8)}\sum_{m=0}^{\infty}\sum_{k=0}^{\infty}\alpha_{k,m}\beta_{k,m}e^{-4(k(k+m+7)+2m)t}\int_{0}^{\pi} {h_m(\eta)}^2 \sin^6{\eta} d\eta\times \int_{0}^{\frac{\pi}{2}}{(\cos r)}^{2m+7}{P^{3,m+3}_k (\cos(2r))}^{2}{(\sin r)}^7 dr\\
&=\frac{56{\pi}^7}{\Gamma(8)}\sum_{m=0}^{\infty}\sum_{k=0}^{\infty}\frac{\alpha_{k,m}\beta_{k,m}e^{-4(k(k+m+7)+2m)t}}{2k+m+7}\times \frac{\sqrt{\pi} \Gamma(7/2)}{\Gamma(4)}\frac{6!m!}{(m+5)!}\frac{\Gamma(k+4)\Gamma(k+m+4)}{\Gamma(k+m+7)\Gamma(k+1)}.
\end{align*}

Above, we used 
\begin{equation*}
\int_{0}^{{\pi}} {h_m(\eta)}^2 \sin^6{\eta} d\eta=\frac{\sqrt{\pi} \Gamma(7/2)}{\Gamma(4)(2m+6)}\frac{6!m!}{(m+5)!},
\end{equation*}
(for example, see the Corollary 9.4.3 in \cite{JF08}).

From
\begin{equation*}
\lim_{t \rightarrow 0}\int_{0}^{\pi}\int_{0}^{\frac{\pi}{2}} p_t f d\mu = f(0,0),
\end{equation*}
we obtain the desired term  $\alpha_{k,m}=\frac{96}{\pi^{8}}(m+3)(2k+m+7)\left(\begin{array}{c}
k+m+6 \\ k+m+3
\end{array}\right)\left(\begin{array}{c}
m+5 \\ m
\end{array}\right)$ and the proof is over.
\end{proof}	
	
\end{prop}
As an immediate corollary for the spectral expansion of the heat kernel, one obtains the spectrum of the sub-Laplacian.

\begin{coro}
	The spectrum of $-{L}$ is given by $\{4(k(k+m+7)+2m) : m,k\geq 0 \}$  and its first non zero eigenvalue is $8$.
\end{coro}

\begin{rema}
	One can compare the spectrum of $L$ of the octonionic $15$-dimensional sphere with the spectrum of the sub-Laplacian of the $9$-dimensional CR-sphere 
	\begin{equation*}
	\mathbb{S}^1 \rightarrow \mathbb{S}^9 \rightarrow \mathbb{CP}^4,
	\end{equation*}
	and the spectrum of the sub-Laplacian of the $11$-dimensional quaternionic-sphere
\begin{equation*}
\mathbb{SU}^2 \rightarrow \mathbb{S}^{11} \rightarrow \mathbb{HP}^2,
\end{equation*}	
which are 
\begin{equation*}
\{4(k(k+m+4)+2m) : m,k\geq 0. \}
\end{equation*}
\begin{equation*}
\{4(k(k+m+5)+2m) : m,k\geq 0. \}
\end{equation*}
given in \cite{FBJW13} and \cite{FBL14} respectively.

\end{rema}

\begin{prop}\label{prop:3}
Let $p_t^{Q}$ and $p_t$ denote the subelliptic heat kernels on the $11$-dimensional quaternionic sphere $\mathbb{S}^{11}$ and the 15-dimensional octonionic sphere $\mathbb{S}^{15}$ respectively. Then for $r\in [0,\frac{\pi}{2}),\eta\in [0,\pi)$, 
\begin{equation}\label{eq:1}
p_t(r,\eta)=\frac{192e^{16t}}{{\pi}^2\cos^2 r}\left(\frac{1}{\sin^2 \eta}\frac{\partial^2}{\partial \eta^2}p_t^{Q}-\frac{\cos \eta}{\sin^3 \eta}\frac{\partial}{\partial \eta}p_t^{Q}\right).
\end{equation}
\begin{proof}
From the Rodrigues formula (for example, see the proposition 9.4.1 in \cite{JF08}), one can verify that 
\begin{equation}\label{eq:5}
\frac{1}{6}\left(\frac{1}{\sin^2 \eta}\frac{\partial^2}{\partial \eta^2}-\frac{\cos \eta}{\sin^3 \eta}\frac{\partial}{\partial \eta}\right)\frac{\sin(m+1)\eta}{\sin \eta}=\left(\begin{array}{c}
m+3 \\ m-2
\end{array}\right) h_{m-2}(\eta).
\end{equation}

On the other hand, from \cite{FBL14}, on the $11$-dimensional quaternionic sphere $\mathbb{S}^{11}$, the spectral decomposition of the quaternionic subelliptic heat kernel $p_t^{Q}(r,\eta)$ is known:
\begin{equation*}
p^{Q}_t(r,\eta)=\sum_{m=0}^{\infty}\sum_{k=0}^{\infty}\beta_{k,m} e^{-4(k(k+m+5)+2m)t}\frac{\sin (m+1)\eta}{\sin \eta}\cos^m r P^{3,m+1}_k(\cos 2r),
\end{equation*}
where 
\begin{equation*}
\beta_{k,m}=\frac{\Gamma(4)}{2\pi^{6}}(2k+m+5)(m+1)\left(\begin{array}{c}
k+m+4 \\ k+m+1
\end{array}\right).
\end{equation*}

Note that the octonionic subelliptic heat kernel $p_t(r,\eta)$ in the previous proposition which was given by:
\begin{equation*}
p_t(r,\eta)=\sum_{m=0}^{\infty}\sum_{k=0}^{\infty}\alpha_{k,m} h_m(\eta)e^{-4(k(k+m+7)+2m)t}\cos^m r P^{3,m+3}_k(\cos 2r),
\end{equation*}
where $\alpha_{k,m}=\frac{96}{\pi^{8}}(m+3)(2k+m+7)\left(\begin{array}{c}
k+m+6 \\ k+m+3
\end{array}\right)\left(\begin{array}{c}
m+5 \\ m
\end{array}\right)$.

From those two expressions of the heat kernels with \eqref{eq:5}, we can easily deduce that 
\begin{equation*}
p_t(r,\eta)=\frac{192e^{16t}}{{\pi}^2\cos^2 r}\left(\frac{1}{\sin^2 \eta}\frac{\partial^2}{\partial \eta^2}p_t^{Q}-\frac{\cos \eta}{\sin^3 \eta}\frac{\partial}{\partial \eta}p_t^{Q}\right).
\end{equation*}
\end{proof}

\end{prop}

\section{Integral representation of the subelliptic heat kernel}

Since $\widetilde{L}=\widetilde{\triangle_{S^{15}}}-\widetilde{\triangle_{S^7}}$, and $\widetilde{L}$ commutes with $\widetilde{\triangle_{S^7}}$ we formally have
\begin{equation*}
e^{t\widetilde{L}}=e^{-t\widetilde{\triangle_{S^7}}}e^{t\widetilde{\triangle_{S^{15}}}}.
\end{equation*}
If we denote by $q_t$ the heat kernel of the heat semigroup $e^{t\widetilde{\triangle_{S^{15}}}}$, then the subelliptic heat kernel $p_t(r,\eta)$ can be obtained by applying the heat semigroup $e^{-t\widetilde{\triangle_{S^7}}}$ on $q_t$, i.e.
\begin{equation*}
p_t(r,\eta)=(e^{-t\widetilde{\triangle_{S^7}}} q_t)(r,\eta).
\end{equation*}
Thus once one knows an integral expression of the heat semigroup $e^{-t\widetilde{\triangle_{S^7}}}$, then one can deduce the integral representation of $p_t(r,\eta)$. Now we have the proposition~\ref{prop:3}, thus we can deduce the integral representation of subelliptic heat kernel $p_t(r,\eta)$ on $\mathbb{S}^{15}$ from the integral representation of the quaternionic subelliptic heat kernel $p_t^{Q}(r,\eta)$ on $\mathbb{S}^{11}$. We now make those heuristic considerations precise.

 Let $q_t$ be the Riemannian radial heat kernel on $\mathbb{S}^{11}$. For later use, we record here that:

(1) The spectral decomposition of $q_t$ is given by 
\begin{equation*}
q_t(\cos \delta)=\frac{\Gamma(5)}{2{\pi}^{6}}\sum_{m=0}^{\infty}(m+5)e^{-m(m+10)t}C^{5}_m(\cos \delta),
\end{equation*}
where $\delta$ is the Riemannian distance from the north pole and 
\begin{equation*}
C^{5}_m(x)=\frac{(-1)^m}{2^m}\frac{\Gamma(m+10)\Gamma(\frac{15}{2})}{\Gamma(10)\Gamma(m+1)\Gamma(m+\frac{11}{2})}\frac{1}{(1-x^2)^{\frac{5}{2}}}\frac{d^m}{d x^m}(1-x^2)^{m+9/2}
\end{equation*}
is a Gegenbauer polynomial. 

(2) Another well known expression (see (3.7) in \cite{FBJW13}) of $q_t(\cos \delta)$ which is useful for the computation of small-time asymptotics is the formula
\begin{equation}\label{heat kernel s11}
q_t(\cos \delta)=e^{25 t}\left(-\frac{1}{2\pi \sin \delta }\frac{\partial}{\partial \delta}\right)^{5}V,
\end{equation}
where $V(t,\delta)=\frac{1}{\sqrt{4\pi t}}\sum_{k\in \mathbb{Z}}e^{-\frac{(\delta-2k\pi)^2}{4t}}$. 

Formula \eqref{heat kernel s11} also  shows that $\delta \to q_t(\cos \delta)$ can analytically be extended to $\delta \in \mathbb C$.

\begin{prop}\label{prop:4}
For $r\in [0,\frac{\pi}{2}),\eta\in [0,\pi)$, we have:
\begin{equation}\label{eq:2}
p_t(r,\eta)=\frac{48e^{15t}}{{\pi}^2\sqrt{\pi t}\cos^2 r}\int_{0}^{\infty}\frac{g_t(\eta,y)}{t^2} e^{{-}\frac{y^2+{\eta}^2}{4t}}q_t(\cos r \cosh y)\sinh y dy,
\end{equation}
where 
\begin{align*}
g_t(\eta,y)&=\csc ^3(\eta)\left( \cos \left(\frac{\eta y}{2 t}\right)2 y (\eta-3 t \cot \eta)\right.\\
&+\left.\sin \left(\frac{\eta y}{2 t}\right) \left(8 t^2 \cot ^2 \eta+4 t^2 \csc ^2 \eta-6 t \eta \cot \eta+2 t+\eta^2-y^2\right)\right).
\end{align*}
\begin{proof}
Note that because of the rapid decay of the integrand and thanks to formula \eqref{heat kernel s11}, we can differentiate under the integral sign. Now from the Proposition~\ref{prop:3}, we know 	
\begin{equation*}
p_t(r,\eta)=\frac{192e^{16t}}{{\pi}^2\cos^2 r}\left(\frac{1}{\sin^2 \eta}\frac{\partial^2}{\partial \eta^2}p_t^{Q}-\frac{\cos \eta}{\sin^3 \eta}\frac{\partial}{\partial \eta}p_t^{Q}\right).
\end{equation*}
From the proposition 2.7 in \cite{FBJW14}, we have
\begin{equation*}
p^{Q}_t(r,\eta)=\frac{e^{-t}}{\sqrt{\pi t}}\int_{0}^{\infty}\frac{\sinh y \sin{\frac{\eta y}{2t}}}{\sin \eta} e^{{-}\frac{y^2+{\eta}^2}{4t}}q_t(\cos r \cosh (y)dy,
\end{equation*}
Plugging in those two ingredients gives the desired result. 
\end{proof}		
		
\end{prop}

\section{The Green function of the operator $-L+40$}

An interesting consequence of the proposition~\ref{prop:3} is the exact computation of the Green function of the operator $-L+40$.
\begin{prop}
For $r\in [0,\frac{\pi}{2}),\eta\in [0,\pi)$, the Green function of the operator $-L+40$ is given by:
\begin{equation*}
G(r,\eta)=\frac{1}{\pi^{8}}\frac{2304}{\left(-2 \cos r \cos \eta+\cos ^2 r+1\right)^5}.
\end{equation*}

\begin{proof}
From the Theorem 2.9 in \cite{FBJW14}, in the case of $\mathbb{S}^{11}$ one has
\begin{equation*}
\int_{0}^{\infty}p^{Q}_{t}(r,\eta)e^{-24t}dt=\frac{1}{4\pi^{6}(1-2\cos r \cos \eta+\cos^2 r)^3}.
\end{equation*}
Using \eqref{eq:1}, tedious computations then yield
\begin{equation*}
\int_{0}^{\infty}p_{t}(r,\eta)e^{-40t}dt=\frac{1}{\pi^{8}}\frac{2304}{\left(-2 \cos r \cos \eta+\cos ^2 r+1\right)^5}.
\end{equation*}
\end{proof}
\end{prop}

\begin{remark}
A similar computation for the complex Hopf fibration on the sphere $\mathbb S^{2n+1}$  yields the Green function of the operator $-L+n^2$, see Proposition 3.4 in \cite{FBJW13}. The operator $-L+n^2$ is the conformal sub-Laplacian on $\mathbb S^{2n+1}$, see  \cite{Geller} and also Theorem 2.1 in \cite{MR3737629} for its relation to the sub-Laplacian of the Heisenberg group. In our case of the octonionic fibration, it would be interesting to interpret $-L+40$ as a "conformal octonionic" sub-Laplacian. In particular, study its relation to the sub-Laplacian on the octonionic Heisenberg group and interpret the number 40. We let this for possible further research.
\end{remark}

\section{Heat kernel small-time asymptotics}

Another advantage of the integral representation \eqref{eq:2} is to deduce the small time asymptotics of the subelliptic heat kernel $p_t$. 
The keypoint is to use the representation \eqref{eq:2} and the exact formula \eqref{heat kernel s11} for $q_t$ which provide uniform estimates for the remainder terms at any order.

The methods to obtain the exact asymptotics are similar to the methods used in \cite{MR2545862, FBJW13,FBJW14}, so we will omit some of the technical justifications. Those methods, in particular the use of steepest descent method, originally go back to \cite{MR1776501} who thoroughly studied the case of the subelliptic heat kernel on the Heisenberg group. Justifications are also written in great details for the sub-Laplacian of the complex Hopf fibration in section 6 of \cite{MR3130140}.

We first note the following uniform small time asymptotics for the Riemannian heat kernel on $\mathbb{S}^{11}$ follows from \eqref{heat kernel s11}
\begin{equation}\label{eq:3a}
q_t(\cos \delta)=\frac{1}{(4\pi t)^{11/2}}\left(\frac{\delta}{\sin \delta}\right)^{5}e^{-\frac{\delta^2}{4t}}\left(1+\left(25-\frac{20(\sin \delta - \delta \cos \delta)}{\delta^2 \sin \delta}\right)t+t^2 R_1(t,\delta) \right).
\end{equation}
The term $R_1(t,\delta) $ is uniformly bounded when $t \to 0$ on any interval $\delta \in [0, \pi-\varepsilon]$, $0<\varepsilon<\pi$. We also deduce from \eqref{heat kernel s11} the uniform small time asymptotics
\begin{equation}\label{eq:3}
q_t(\cosh \delta)=\frac{1}{(4\pi t)^{11/2}}\left(\frac{\delta}{\sinh \delta}\right)^{5}e^{\frac{\delta^2}{4t}}\left(1+\left(25+\frac{20(\sinh \delta - \delta \cosh \delta)}{\delta^2 \sinh \delta}\right)t+t^2 R_2(t,\delta) \right).
\end{equation}

The term $R_2(t,\delta) $ is uniformly bounded when $t \to 0$ on $[0,+\infty)$.

By applying \eqref{eq:2}, one can then deduce the small time asymptotics of the subelliptic heat kernel. 

\begin{prop}
	When $t \rightarrow 0$,
\begin{equation*}
p_t(0,0)=\frac{1}{5\times2^{10}\pi^{8}t^{11}}(A+Bt+O(t^2)),
\end{equation*}	
where 
\begin{align*}
A&=\int_{0}^{\infty}y^5\left(\frac{y}{\sinh y}\right)^5dy,\\
B&=20\int_{0}^{\infty}\left(y^5\left(2+\frac{\sinh y-y\cosh y}{y^2 \sinh y} \right)-y^3 \right)\left(\frac{y}{\sinh y}\right)^5dy.
\end{align*}
\end{prop}

\begin{proof}
From \eqref{eq:2}, we have 
\begin{equation*}
p_t(0,0)=\frac{48e^{15t}}{{\pi}^2\sqrt{\pi t}}\int_{0}^{\infty}\frac{g_t(0,y)}{t^2} e^{{-}\frac{y^2}{4t}}q_t( \cosh y)\sinh ydy,
\end{equation*}	
where $\frac{g_t(0,y)}{t^2}$ is given by 
\begin{equation}\label{eq:g_t}
\frac{y \left(64 t^4+120 t^3-20 t^2 \left(y^2-3\right)-20 t y^2+y^4\right)}{120 t^5}.
\end{equation}
Plug in \eqref{eq:3}, we have that 
\begin{equation}\label{eq:4}
p_t(0,0)=\frac{3e^{15t}}{2^{7}\pi^{8}t^6}\int_{0}^{\infty}\frac{g_t(0,y)}{t^2} \left(\frac{y}{\sinh y}\right)^5 \left(1+\left(25+\frac{20(\sinh y-y\cosh y)}{y^2 \sinh y} \right)t \right)dy+O(t^{-9}).
\end{equation}
Thus we obtain the small time asymptotic as follows:
\begin{equation*}
p_t(0,0)=\frac{1}{5\times2^{10}\pi^{8}t^{11}}(A+Bt+O(t^2)),
\end{equation*}	
where $A,B$ are constants as stated in the proposition. 
\end{proof}

The small time behavior of the subelliptic heat kernel on the vertical cut-locus, namely the points $(0,\eta)$ that can be achieved by flowing along vertical vector fields is quite different. We can deduce it is by differentiating the small time estimate of $p^{Q}_t(0,\eta)$.

\begin{prop}
For $\eta \in (0,\pi)$, $t \rightarrow 0$, 
\begin{equation*}
p_t(0,\eta)=\frac{1}{\pi^3 2^{12} t^{11}}e^{-\frac{\eta (-\eta+2 \pi )}{4 t}}\csc ^5 \eta \left(\eta^3 \left(-\eta^3+3 \pi  \eta^2-3 \pi ^2 \eta+\pi ^3)(1-\cos 2 \eta \right)+O(t)  \right) .
\end{equation*}


\end{prop}
\begin{proof}
From \eqref{eq:1}, 
\begin{equation*}
p_t(r,\eta)=\frac{192e^{16t}}{{\pi}^2\cos^2 r} \left(\frac{1}{\sin^2 \eta}\frac{\partial^2}{\partial \eta^2}p_t^{Q}-\frac{\cos \eta}{\sin^3 \eta}\frac{\partial}{\partial \eta}p_t^{Q}\right).
\end{equation*}	
From  proposition 2.11 in \cite{FBJW14} one has

\begin{equation*}
p^{Q}_t(0,\eta)=\frac{1}{3\pi 2^{15} t^{9}\sin \eta }(\pi-\eta)\eta^{3}e^{-\frac{2\pi \eta-\eta^2}{4t}} (1+O(t)),
\end{equation*}	
One can justify the differentiation of this asymptotics using the exact analytic formula for $p^{Q}_t$ which is given in proposition 2.7 in \cite{FBJW14} . Then the computation gives the result. 

\end{proof}

For last two propositions, we will apply the Laplace method and the steepest descent method. 

\begin{prop}\label{prop:8}
For $r\in (0,\frac{\pi}{2})$, we have 
\begin{equation*}
p_t(r,0)\sim_{t\rightarrow 0}\frac{3}{2^7}\frac{1}{( \pi t)^{15/2}\cos^2 r}e^{-\frac{r^2}{4t}}\left(\frac{r}{\sin r}\right)^5 \left(\frac{1}{1-r\cot r}\right)^{7/2}.
\end{equation*}
\begin{proof}
By \eqref{eq:2}, we have 
\begin{equation*}
p_t(r,0)=\frac{48e^{15t}}{{\pi}^2\sqrt{\pi t}\cos^2 r}\int_{-\infty}^{\infty}\frac{g_t(0,y)}{t^2} e^{{-}\frac{y^2}{4t}}q_t(\cos r \cosh y) \sinh ydy,
\end{equation*}
where $\frac{g_t(0,y)}{t^2}$ is given by \eqref{eq:g_t} and $q_t$ is the Riemannian heat kernel on $\mathbb{S}^{11}$.
By plugging in \eqref{eq:3}, we obtain that 
\begin{equation*}
p_t(r,0)\sim_{t\rightarrow 0}\frac{45}{32\pi^8 \cos^2 r}\frac{1}{t^{11}}(J_1(t)+J_2(t)),
\end{equation*}
where 
\begin{equation*}
J_1(t)=\int_{\cosh y \leq \frac{1}{\cos r}}e^{-{\frac{y^2+(\arccos(\cos r \cosh y))^2}{4t}}}\frac{g_t(0,y)}{t^2} \left(\frac{\arccos(\cos r \cosh y) }{\sqrt{1-\cos^2 r \cosh^2 y }} \right)^5 \sinh ydy,
\end{equation*}
and
\begin{equation*}
J_2(t)=\int_{\cosh y \geq \frac{1}{\cos r}}e^{-{\frac{y^2-(\cosh^{-1}(\cos r \cosh y))^2}{4t}}}\frac{g_t(0,y)}{t^2} \left(\frac{\cosh^{-1}(\cos r \cosh y) }{\sqrt{\cos^2 r \cosh^2 y }-1} \right)^5 \sinh ydy,
\end{equation*}
The idea is to analyze $J_1(t)$ and $J_2(t)$ by Laplace method. Furthermore, since we are interested in the asymptotic behavior when $t\rightarrow 0$, it suffices to consider the dominant term of $\frac{g_t(0,y)}{t^2}$ only, which is $\frac{y^5}{120}$.

Notice that in $[-\cosh^{-1}(\frac{1}{\cos r}), \cosh^{-1}(\frac{1}{\cos r})]$, $f(y)=y^2+(\arccos(\cos r \cosh y))^2$ has a unique minimum at $y=0$, where 
\begin{equation*}
f''(0)=2(1-r\cot r).
\end{equation*}
Hence by Laplace method, we can easily obtain that 
\begin{equation*}
J_1(t)\sim_{t\rightarrow 0}240\sqrt{\pi}e^{-\frac{r^2}{4t}}\left(\frac{r}{\sin r}\right)^5 \left(\frac{1}{1-r\cot r}\right)^{7/2} t^{7/2}.
\end{equation*}
On the other hand, on $(-\infty,-\cosh^{-1}(\frac{1}{\cos r}))\cup(\cosh^{-1}(\frac{1}{\cos r}),\infty)$, the function $y^2-(\cosh^{-1}(\cos r \cosh y))^2$ has no minimum, which implies that $J_2(t)$ is negligible with respect to $J_1(t)$ in small $t$. This finishes the proof. 
\end{proof}
\end{prop}

For the case $(r,\eta)$ with $r\neq 0$, the Laplace method no longer works, we need to use the steepest descent method. 

\begin{prop}\label{prop:9} 
Let $r\in (0,\frac{\pi}{2})$, $\eta \in [0,\pi)$. Then when $t \rightarrow 0$,
\begin{equation}\label{eq:7}
p_t(r,\eta)\sim_{t\rightarrow 0}\frac{3\sin \varphi(r,\eta)}{2^7 (\pi t)^{15/2}}\frac{\left( \eta^2+{\varphi(r,\eta)}^2\right)}{\sin^3 \eta \cos^2 r\sin r}  \frac{(\arccos u(r,\eta))^{5} }{\sqrt{1-\frac{u(r,\eta)\arccos u(r,\eta) }{\sqrt{1-u^2(r,\eta)}}}}\frac{e^{-{\frac{(\varphi(r,\eta)+\eta)^2\tan^2 r}{4t\sin^2(\varphi(r,\eta)) }}}}{(1-u^2(r,\eta))^2}
\end{equation}
where $u(r,\eta)=\cos r \cos \varphi(r,\eta)$ and $\varphi(r,\eta)$ is the unique solution in $[0,\pi]$ to the equation
\begin{equation}\label{eq:6}
\varphi(r,\eta)+\eta=\cos r \sin \varphi(r,\eta)\frac{\arccos(\cos \varphi(r,\eta)\cos r)}{\sqrt{1-\cos^2 r \cos^2 \varphi(r,\eta)}}.
\end{equation}

\begin{proof}
From the Proposition~\ref{prop:4}, we can rewrite
\begin{equation*}
p_t(r,\eta)=\frac{24e^{15t}}{{(\pi t)}^{2.5}\cos^2 r \sin^3 \eta}\int_{-\infty}^{\infty}l_t(\eta,y) q_t(\cos r \cosh y)\sinh y dy,
\end{equation*}
where $l_t(\eta,y)$ is given by 
\begin{equation*}
e^{\frac{(\eta+iy)^2}{4t}} \left( y (\eta-3 t \cot \eta) +\frac{1}{2i}(8 t^2 \cot ^2 \eta+4 t^2 \csc ^2 \eta-6 t \eta \cot \eta+2 t+\eta^2-y^2)\right)
\end{equation*}

\begin{equation*}
+e^{\frac{(\eta-iy)^2}{4t}} \left( y (\eta-3 t \cot \eta) -\frac{1}{2i}(8 t^2 \cot ^2 \eta+4 t^2 \csc ^2 \eta-6 t \eta \cot \eta+2 t+\eta^2-y^2)\right).
\end{equation*}
{\tiny }

Since we consider when $t\rightarrow 0$, we only need to consider the dominant terms of $l_t(\eta,y)$, thus we may assume that $l_t(\eta,y)$ can be written as 
\begin{equation*}
e^{-\frac{(y-i\eta)^2}{4t}} \left( y \eta +\frac{1}{2i}(\eta^2-y^2)\right)+e^{-\frac{(y+i\eta)^2}{4t}} \left( y \eta -\frac{1}{2i}(\eta^2-y^2)\right).
\end{equation*}
Hence we obtain that 
\begin{equation*}
p_t(r,\eta)\sim_{t\rightarrow 0}\frac{3}{2^{8} \pi^{8} t^8}(A+B),
\end{equation*}
where
\begin{equation*}
A=\int^{\infty}_{-\infty}\frac{\sinh y}{\cos^2 r \sin^3 \eta} \left( y \eta +\frac{1}{2i}(\eta^2-y^2)\right)e^{-{\frac{(y-i\eta)^2+(\arccos(\cos r \cosh y))^2}{4t}}} \left(\frac{\arccos(\cos r \cosh y) }{\sqrt{1-\cos^2 r \cosh^2 y }} \right)^5 dy,
\end{equation*}
and
\begin{equation*}
B=\int^{\infty}_{-\infty}\frac{\sinh y}{\cos^2 r\sin^3 \eta}\left( y \eta -\frac{1}{2i}(\eta^2-y^2)\right)e^{-{\frac{(y+i\eta)^2+(\arccos(\cos r \cosh y))^2}{4t}}} \left(\frac{\arccos(\cos r \cosh y) }{\sqrt{1-\cos^2 r \cosh^2 y }} \right)^5 dy.
\end{equation*}
For the small time asymptotic of $B$: By applying the steepest descent method, we can constraint the integral on the strip $|Re(y)|<\cosh^{-1}(\frac{1}{\cos r})$ where, due to the result in \cite{FBJW13} (Lemma 3.9),
\begin{equation*}
f(y)=(y+i\eta)^2+(\arccos(\cos r \cosh y))^2
\end{equation*}
has a critical point at $i\varphi(r,\eta)$, where $\varphi(r,\eta)$ is the unique solution in $[0,\pi]$ to the equation \eqref{eq:6} and 
\begin{equation*}
f''(i\varphi(r,\eta))=\frac{2\sin^2 r}{1-u(r,\eta)^2}\left(1-\frac{u(r,\eta)\arccos u(r,\eta)}{\sqrt{1-u^2(r,\eta)}}\right)
\end{equation*}
is positive, where $u(r,\eta)=\cos r \cos \varphi(r,\eta)$. Note that 
\begin{align*}
f(i\varphi(r,\eta))&=(-\varphi(r,\eta)+\eta)^2+\arccos^2(\cos r \cos\varphi(r,\eta)) \\
&=(\varphi(r,\eta)+\eta)^2\left(-1+\frac{1-\cos^2 r \cos^2 \varphi(r,\eta)}{\cos^2 r \sin^2 \varphi(r,\eta)}\right)=(\varphi(r,\eta)+\eta)^2\frac{\tan^2 r}{\sin^2 \varphi(r,\eta)}.
\end{align*}
Thus for a sufficiently small $t>0$, $B$ has the following estimates:
\begin{equation*}
B\sim_{t\rightarrow 0}\frac{\sqrt{4\pi t}\sin \varphi(r,\eta)}{\cos^2 r\sin r \sin^3 \eta}\left( \varphi(r,\eta) \eta +\frac{1}{2}(\eta^2+{\varphi(r,\eta)}^2\right)\frac{(\arccos u(r,\eta))^{5} }{\sqrt{1-\frac{u(r,\eta)\arccos u(r,\eta) }{\sqrt{1-u^2(r,\eta)}}}}\frac{e^{-{\frac{(\varphi(r,\eta)+\eta)^2\tan^2 r}{4t\sin^2(\varphi(r,\eta)) }}}}{(1-u^2(r,\eta))^2}.
\end{equation*}
To estimate $A$, we denote
\begin{equation*}
g(y)=(y-i\eta)^2+(\arccos(\cos r \cosh y))^2.
\end{equation*}
Then easy computations show that $g(y)$ has a critical point at $-i\varphi(r,\eta)$ where $\varphi(r,\eta)$ is as described in \eqref{eq:6}. Thus
\begin{equation*}
A\sim_{t\rightarrow 0}-\frac{\sqrt{4\pi t}\sin \varphi(r,\eta)}{\cos^2 r\sin r\sin^3 \eta}\left(\varphi(r,\eta) \eta -\frac{1}{2}(\eta^2+{\varphi(r,\eta)}^2\right)\frac{(\arccos u(r,\eta))^{5} }{\sqrt{1-\frac{u(r,\eta)\arccos u(r,\eta) }{\sqrt{1-u^2(r,\eta)}}}}\frac{e^{-{\frac{(\varphi(r,\eta)+\eta)^2\tan^2 r}{4t\sin^2(\varphi(r,\eta)) }}}}{(1-u^2(r,\eta))^2}.
\end{equation*}
By putting $A\sim_{t\rightarrow 0}$ and $B\sim_{t\rightarrow 0}$ together, we obtain \eqref{eq:7}.

\end{proof}

\end{prop}

\begin{rema}
In proposition~\ref{prop:9}, if we let $\eta=0$, then the equation \eqref{eq:6} has a unique solution at $\varphi=0$ and 
\begin{equation*}
\lim_{\eta \rightarrow 0}\frac{\varphi(r,\eta)}{\eta}=-\frac{1}{1-r\cot r}.
\end{equation*}	
Thus equation \eqref{eq:7} gives that 
\begin{equation*}
p_t(r,0)\sim_{t\rightarrow 0}\frac{3}{2^7}\frac{1}{( \pi t)^{15/2}\cos^2 r}e^{-\frac{r^2}{4t}}\left(\frac{r}{\sin r}\right)^5 \left(\frac{1}{1-r\cot r}\right)^{7/2}
\end{equation*}
which agrees with the result in proposition~\ref{prop:8}.	
\end{rema}

\begin{rema}
	By symmetry, the sub-Riemannian distance from the north pole to any point on $\mathbb{S}^{15}$ only depends on $r$ and $\eta$. If we denote it by $d(r,\eta)$, then from the previous propositions, using the fact that from \cite{Le1,Le2} one has
	\[
	d^2(r,\eta)=-\lim_{t \to 0} 4t \ln p_t(r,\eta),
	\]
	one deduces:
	
(1) For $\eta\in [0,\pi)$,
\begin{equation*}
d^2(0,\eta)=2\pi \eta-\eta^2.
\end{equation*}	

(2) For $\eta \in [0,\pi)$, $r\in[0,\frac{\pi}{2})$,
\begin{equation*}
d^2(r,\eta)=\frac{(\varphi(r,\eta)+\eta)^2\tan^2 r}{\sin^2(\varphi(r,\eta)) }.
\end{equation*}

(3) For  $r\in[0,\frac{\pi}{2})$,
\begin{equation*}
d(r,0)=r.
\end{equation*}
In particular, the sub-Riemannian diameter of $\mathbb{S}^{15}$ is $\pi$.	
	
\end{rema}

\bibliographystyle{spmpsci}
\bibliography{reference}

\end{document}